\documentclass[11pt]{article}%
\usepackage{amsfonts}
\usepackage{graphicx}
\usepackage{amsmath}
\usepackage{amssymb}
\usepackage{hyperref}%
\providecommand{\U}[1]{\protect\rule{.1in}{.1in}}
\newtheorem{theorem}{Theorem}

\newtheorem{corollary}[theorem]{Corollary}

\newtheorem{example}[theorem]{Example}

\newtheorem{proposition}[theorem]{Proposition}

\newenvironment{proof}[1][Proof]{\textbf{#1.} }{\ \rule{0.5em}{0.5em}}
\begin{document}

\author{Ognyan Kounchev
\and Hermann Render
\and Tsvetomir Tsachev}
\title{Inequalities for exponential polynomials with applications to moment sequences
\thanks{The work of all authors was funded under project KP-06-N52-1 with
Bulgarian NSF.} }
\date{}
\maketitle

\begin{abstract}
Let $\Phi_{\Lambda_{n}}$ be the unique solution of the differential operator
$L=\prod_{j=0}^{n}\left(  \frac{d}{dx}-\lambda_{j}\right)  $ such that
$\Phi_{\Lambda_{n}}^{\left(  j\right)  }\left(  0\right)  =0$ for
$j=0,...,n-1,$ and $\Phi_{\Lambda_{n}}^{\left(  n\right)  }\left(  0\right)
=1.$ Assume that $\Phi_{\Lambda_{n}}$ is real-valued and $\Phi_{\Lambda_{n}%
}^{\left(  n+1\right)  }\left(  x\right)  \geq0$ for all $x\in\left[
0,B\right]  .$ Then, if a polynomial $R\left(  x\right)  =%
{\displaystyle\sum_{k=0}^{n}}
a_{k}x^{k}$ is non-negative on the interval $\left[  0,B\right]  ,$ the
inequality
\[%
{\displaystyle\sum_{k=0}^{n}}
a_{k}k!\Phi_{\Lambda_{n}}^{\left(  n-k\right)  }\left(  x\right)  \geq
R\left(  x\right)
\]
holds for $x\in\left[  0,B\right]  $. From this we derive several interesting
inequalities for exponential polynomials. An important consequence is that for
a non-negative measure $\mu$ over the interval $\left[  a,b\right]  $ with
$b-a<B$ the sequence defined by
\[
s_{k}:=\int_{a}^{b}k!\Phi_{\Lambda_{n}}^{\left(  n-k\right)  }\left(
x-a\right)  d\mu\left(  x\right)
\]
for $k=0,...,n$ is a moment sequence, i.e. there exists a non-negative measure
$\nu$ with support in $\left[  a,b\right]  $ such that $s_{k}=\int_{a}%
^{b}\left(  t-a\right)  ^{k}d\nu\left(  t\right)  $ for $k=0,....,n.$

\textbf{AMS Classification, MSC2010}: 44A60, 30E05, 42A82. 

\end{abstract}


Ognyan Kounchev 

Institute of Mathematics and Informatics, Bulgarian

Academy of Sciences, Acad. G. Bonchev str., block 8, 1113 Sofia,

Bulgaria, email: kounchev@math.bas.bg

Hermann Render 

School of Mathematics and Statistics, University College Dublin,  Belfield,
Dublin 4, Ireland,

hermann.render@ucd.ie

Tsvetomir Tsachev 

Institute of Mathematics and Informatics, Bulgarian

Academy of Sciences, Acad. G. Bonchev str., block 8, 1113 Sofia,

Bulgaria, tsachev@math.bas.bg

\section{Introduction}

Let $L$ be a linear differential operator $L$ of order $n+1$ with constant
coefficients, so $L$ is of the form
\begin{equation}
L=L_{\left(  \lambda_{0},...,\lambda_{n}\right)  }=\prod_{j=0}^{n}\left(
\frac{d}{dx}-\lambda_{j}\right)  \label{neuLa}%
\end{equation}
where $\lambda_{0},...,\lambda_{n}$ are arbitrary complex numbers. Any
solution of the equation $L_{n}f=0$ is called an \emph{exponential polynomial}
or $L$\emph{-polynomial. }For example, in the case of pairwise different
$\lambda_{j},j=0,\ldots,n,$ the set of all solutions is the linear span
generated by the functions $e^{\lambda_{0}x},e^{\lambda_{1}x},\ldots
,e^{\lambda_{n}x}.$ In the case when $\lambda_{j}$ occurs $m_{j}$ times in
$\Lambda_{n}=\left(  \lambda_{0},\ldots,\lambda_{n}\right)  ,$ a basis of the
space of all solutions is given by the linearly independent functions
$x^{s}e^{\lambda_{j}x}$ for $s=0,1,\ldots,m_{j}-1.$ We call $\lambda
_{0},\ldots,\lambda_{n}$ the \emph{exponents} or \emph{frequencies} (see
e.g.\ Chapter 3 in \cite{BeGa95}) and use the notation
\begin{equation}
E_{\Lambda_{n}}:=\left\{  f\in C^{\infty}\left(  \mathbb{R},\mathbb{C}\right)
:L_{n}f=0\right\}  \label{solutspace}%
\end{equation}
Exponential polynomials are used in many areas of mathematics: for classical
applications in complex analysis and number theory we refer to \cite{BeGa95}.
Our interest in exponential polynomials originates from applications in
polyharmonic function theory,  interpolation theory and multidimensional
moment problem, see \cite{KoRe13}, \cite{KoReMathNach},  \cite{KoReInter2019},
\cite{KoReArkiv2010}, from applications in subdivision schemes, spline
analysis and wavelet analysis, see \cite{dynlevinexp}, \cite{DKLR11}
\cite{KounchevRender2021JCAM}, \cite{McCa91}, \cite{micchelli} and
\cite{Unser}, from applications in CAGD, see \cite{AKR07}, \cite{CMP04},
\cite{CMP14} and \cite{CMP17}, and finally from applications in data analysis
for smoothing data, see \cite{KounchevRenderTsachevBIT} and
\cite{RamsaySilverman}.

Let us recall that for every $\Lambda_{n}=\left(  \lambda_{0},\ldots
,\lambda_{n}\right)  \in\mathbb{C}^{n+1}$ there exists a unique solution
$\Phi_{\Lambda_{n}}\in E_{\Lambda_{n}}$ to the Cauchy problem
\[
\Phi_{\Lambda_{n}}\left(  0\right)  =\ldots=\Phi_{\Lambda_{n}}^{\left(
n-1\right)  }\left(  0\right)  =0\text{ and }\Phi_{\Lambda_{n}}^{\left(
n\right)  }\left(  0\right)  =1,
\]
and $\Phi_{\Lambda_{n}}$ is called sometimes the \emph{fundamental function}
in $E_{\Lambda_{n}}$ (see Proposition 13.12 in \cite{okbook}, or Chapter $9$
in \cite{Sc81}, or \ \cite{micchelli}, or \cite[p. 501]{Karl68}). Here
$f^{\left(  k\right)  }$ denotes the $k$-th derivative of a function $f.$ In
the present paper we pay a special attention to the following basis of
functions
\[
b_{k}\left(  x\right)  =k!\Phi_{\Lambda_{n}}^{\left(  n-k\right)  }\left(
x\right)  \text{ for }k=0,...,n
\]
for $E_{\Lambda_{n}}$: obviously they are linearly independent since the basis
function $b_{j}$ has a zero of order $j$ at $x=0.$ The unexpected factor $k!$
and the reverse numeration $n-k$ are motivated by the fact that in the
polynomial case, i.e. that $\lambda_{0}=\cdots=\lambda_{n}$ $=0,$ the basis
function $b_{k}\left(  x\right)  $ are equal to the power functions $x^{k}$
for $j=0,...,n$.

The main result of the paper is the following:

\begin{theorem}
\label{ThmMain1}Let $\Lambda_{n}=\left(  \lambda_{0},\ldots,\lambda
_{n}\right)  \in\mathbb{C}^{n+1}$ and $B>0.$ Assume that the fundamental
function $\Phi_{\Lambda_{n}}$ is real-valued and that $\Phi_{\Lambda_{n}%
}^{\left(  n+1\right)  }\left(  x\right)  \geq0$ for all $x\in\left[
0,B\right]  .$ Then, if a non-zero polynomial $R\left(  x\right)  =%
{\displaystyle\sum_{k=0}^{n}}
a_{k}x^{k}$ is non-negative on the interval $\left[  0,B\right]  ,$ the
inequality
\[%
{\displaystyle\sum_{k=0}^{n}}
a_{k}k!\Phi_{\Lambda_{n}}^{\left(  n-k\right)  }\left(  x\right)  >R\left(
x\right)
\]
holds for $x\in\left[  0,B\right]  .$
\end{theorem}

One interesting application is the following result:

\begin{theorem}
\label{ThmMain2}Let $\Lambda_{n}=\left(  \lambda_{0},\ldots,\lambda
_{n}\right)  \in\mathbb{C}^{n+1}$ and $B>0.$ Assume that $\Phi_{n}%
:=\Phi_{\Lambda_{n}}$ is real-valued and $\Phi_{n}^{\left(  n+1\right)
}\left(  x\right)  \geq0$ for all $x\in\left[  0,B\right]  .$ Then for any
natural number $k\leq n/2$ the Hankel matrix of size $\left(  k+1\right)
\times\left(  k+1\right)  $ defined by
\[
H_{n,k}\left(  x\right)  =\left(
\begin{array}
[c]{cccc}%
\Phi_{n}^{\left(  n\right)  }\left(  x\right)  & \Phi_{n}^{\left(  n-1\right)
}\left(  x\right)  & .... & k!\Phi_{n}^{\left(  n-k\right)  }\left(  x\right)
\\
\Phi_{n}^{\left(  n-1\right)  }\left(  x\right)  & ... & k!\Phi_{n}^{\left(
n-k\right)  }\left(  x\right)  & \left(  k+1\right)  !\Phi_{n}^{\left(
n-k-1\right)  }\left(  x\right) \\
\vdots &  &  & \vdots\\
k!\Phi_{n}^{\left(  n-k\right)  }\left(  x\right)  & \cdots & \cdots & \left(
2k\right)  !\Phi_{n}^{\left(  n-2k\right)  }\left(  x\right)
\end{array}
\right)
\]
is positive definite, i.e. $p^{T}H_{n,k}\left(  x\right)  p>0$ holds for any
$x\in\left[  0,B\right]  $ and for any non-zero $p\in\mathbb{R}^{k+1}$.
\end{theorem}

Another interesting application concerns the moment problem:

\begin{theorem}
\label{ThmMain3}Let $\Lambda_{n}=\left(  \lambda_{0},\ldots,\lambda
_{n}\right)  \in\mathbb{C}^{n+1}$ and $B>0.$ Assume that $\Phi_{\Lambda_{n}}$
is real-valued and $\Phi_{\Lambda_{n}}^{\left(  n+1\right)  }\left(  x\right)
\geq0$ for all $x\in\left[  0,B\right]  .$ If $\mu$ is non-negative measure
over the interval $\left[  a,b\right]  $ with $b-a<B$ then there exists a
non-negative measure $\nu$ with support in $\left[  a,b\right]  $ such that
\[
\int_{a}^{b}k!\Phi_{\Lambda_{n}}^{\left(  n-k\right)  }\left(  x-a\right)
d\mu\left(  x\right)  =\int_{a}^{b}t^{k}d\nu\left(  t\right)
\]
for $k=0,....,n.$
\end{theorem}

The paper is organized as follows: in Section 2 we will provide the proofs of
Theorem 1 to 3 and an extension of the results for the case $x<0.$ As a
further application we will provide in Section 3 a proof of the inequality
\[
\Phi_{\Lambda_{n}}^{\prime\prime}\left(  x\right)  \Phi_{\Lambda_{n}}\left(
x\right)  \leq\left(  \Phi_{\Lambda_{n}}^{\prime}\left(  x\right)  \right)
^{2}<\frac{n}{n-1}\Phi_{\Lambda_{n}}^{\prime\prime}\left(  x\right)
\Phi_{\Lambda_{n}}\left(  x\right)
\]
for all $x\in\left(  0,B\right)  $ where the upper bound is valid under the
assumption that $\Phi_{\Lambda_{n}}$ is real-valued and $\Phi_{\Lambda_{n}%
}^{\left(  n+1\right)  }\left(  x\right)  \geq0$ for all $x\in\left[
0,B\right]  ,$ and the lower bound under the assumption that $E_{\Lambda_{n}}$
is an extended Chebyshev space on $\left[  0,B\right]  $ (see for details
Section 3). In Section 4 we discuss criteria for the monotonicity of
derivatives of the fundamental function which shows that the assumption
$\Phi_{\Lambda_{n}}^{\left(  n+1\right)  }\left(  x\right)  \geq0$ for all
$x\in\left[  0,B\right]  $ is satisfied in many situations.

\section{The proofs of the main results}

The following simple observation is crucial for our proofs. Note that in the
case $x<0$ we use the convention that $\int_{0}^{x}f\left(  t\right)
dt:=-\int_{x}^{0}f\left(  t\right)  dt$.

\begin{theorem}
\label{ThmMain0}Let $R\left(  x\right)  =%
{\displaystyle\sum_{k=0}^{n}}
a_{k}x^{k}$ be a polynomial of degree $n.$ Then the following identity holds
for real numbers $x$
\[%
{\displaystyle\sum_{k=0}^{n}}
a_{k}k!\Phi_{\Lambda_{n}}^{\left(  n-k\right)  }\left(  x\right)  =R\left(
x\right)  +\int_{0}^{x}R\left(  t\right)  \Phi_{\Lambda_{n}}^{\left(
n+1\right)  }\left(  x-t\right)  dt.
\]

\end{theorem}

\begin{proof}
Assume that $x>0.$ Note that $R^{\left(  k\right)  }\left(  0\right)
=k!a_{k}.$ It suffices to prove for $k=0,....,n$ that%
\begin{equation}
k!\Phi_{\Lambda_{n}}^{\left(  n-k\right)  }\left(  x\right)  =x^{k}+\int
_{0}^{x}t^{k}\Phi_{\Lambda_{n}}^{\left(  n+1\right)  }\left(  x-t\right)  dt.
\label{eqnew0}%
\end{equation}
Indeed, multiply the equation with $a_{k}$ and sum up from $k=0,....,n,$ and
the statement follows. We use induction to prove (\ref{eqnew0}) and we note
that $k=0$
\[
\int_{0}^{x}\Phi_{\Lambda_{n}}^{\left(  n+1\right)  }\left(  x-t\right)
dt=-\Phi_{\Lambda_{n}}^{\left(  n\right)  }\left(  x-t\right)  \mid_{t=0}%
^{x}=\Phi_{\Lambda_{n}}^{\left(  n\right)  }\left(  x\right)  -\Phi
_{\Lambda_{n}}^{\left(  n\right)  }\left(  0\right)  =\Phi_{\Lambda_{n}%
}^{\left(  n\right)  }\left(  x\right)  -1.
\]
Now suppose that the statement is true for $k.$ We use partial integration to
see that
\begin{align*}
&  \int_{0}^{x}t^{k+1}\Phi_{\Lambda_{n}}^{\left(  n+1\right)  }\left(
x-t\right)  dt\\
&  =-t^{k+1}\Phi_{\Lambda_{n}}^{\left(  n\right)  }\left(  x-t\right)
\mid_{t=0}^{x}+\left(  k+1\right)  \int_{0}^{x}t^{k}\Phi_{\Lambda_{n}%
}^{\left(  n\right)  }\left(  x-t\right)  dt\\
&  =-x^{k+1}+\left(  k+1\right)  \int_{0}^{x}t^{k}\Phi_{\Lambda_{n}}^{\left(
n\right)  }\left(  x-t\right)  dt.
\end{align*}
In order to complete the induction proof we just need the fact that
\begin{equation}
\int_{0}^{x}t^{k}\Phi_{\Lambda_{n}}^{\left(  n\right)  }\left(  x-t\right)
dt=k!\Phi_{\Lambda_{n}}^{\left(  n-k-1\right)  }\left(  x\right)  .
\label{eqNew}%
\end{equation}
Now let us show that (\ref{eqNew}) holds: For $k>0$ partial integration shows
that
\[
\int_{0}^{x}t^{k}\Phi_{\Lambda_{n}}^{\left(  n\right)  }\left(  x-t\right)
dt=-t^{k}\Phi_{\Lambda_{n}}^{\left(  n-1\right)  }\left(  x-t\right)
\mid_{t=0}^{x}+k\int_{0}^{x}t^{k-1}\Phi_{\Lambda_{n}}^{\left(  n-1\right)
}\left(  x-t\right)  dt,
\]
hence the recursion formula
\[
\int_{0}^{x}t^{k}\Phi_{\Lambda_{n}}^{\left(  n\right)  }\left(  x-t\right)
dt=k\int_{0}^{x}t^{k-1}\Phi_{\Lambda_{n}}^{\left(  n-1\right)  }\left(
x-t\right)  dt
\]
is true. From this one easily obtain
\begin{align*}
\int_{0}^{x}t^{k}\Phi_{\Lambda_{n}}^{\left(  n\right)  }\left(  x-t\right)
dt  &  =k!\int_{0}^{x}\Phi_{\Lambda_{n}}^{\left(  n-k\right)  }\left(
x-t\right)  dt\\
&  =k!\left(  -\Phi_{\Lambda_{n}}^{\left(  n-k-1\right)  }\left(  x-t\right)
\mid_{t=0}^{x}\right)  =k!\Phi_{\Lambda_{n}}^{\left(  n-k-1\right)  }\left(
x\right)  .
\end{align*}
In the case $x<0$ the result is proved in a very similar fashion.
\end{proof}

\bigskip

\textbf{Proof of Theorem \ref{ThmMain1}:} Since $\Phi_{\Lambda_{n}}^{\left(
n+1\right)  }\left(  x\right)  \geq0$ for all $x\in\left[  0,B\right]  $ and
$R\left(  x\right)  \geq0$ for all $x\in\left[  0,B\right]  $ we see that%
\[
R\left(  t\right)  \Phi_{\Lambda_{n}}^{\left(  n+1\right)  }\left(
x-t\right)  \geq0
\]
for any $x\in\left[  0,B\right]  .$ Since $t\longmapsto\Phi_{\Lambda_{n}%
}^{\left(  n+1\right)  }\left(  x-t\right)  R\left(  t\right)  $ is a non-zero
analytic function it has only finitely many zeros on $\left[  0,B\right]  $
and therefore
\[
\int_{0}^{x}R\left(  t\right)  \Phi_{\Lambda_{n}}^{\left(  n+1\right)
}\left(  x-t\right)  dt>0.
\]
Now the statement is obvious from Theorem \ref{ThmMain0}.

\bigskip

\textbf{Proof of Theorem \ref{ThmMain2}:} Let $p=\left(  p_{0},....,p_{k}%
\right)  \in\mathbb{R}^{k+1}$ non-zero. It follows that
\begin{align*}
p^{t}H_{n,k}p  &  =%
{\displaystyle\sum_{r=0}^{k}}
{\displaystyle\sum_{s=0}^{k}}
p_{r}p_{s}\left(  r+s\right)  !\Phi_{\Lambda_{n}}^{\left(  n-\left(
s+r\right)  \right)  }\left(  x\right) \\
&  =%
{\displaystyle\sum_{j=0}^{2k}}
{\displaystyle\sum_{s+r=j}}
p_{r}p_{s}\left(  r+s\right)  !\Phi_{\Lambda_{n}}^{\left(  n-j\right)
}\left(  x\right)  =%
{\displaystyle\sum_{j=0}^{2k}}
a_{j}j!\Phi_{\Lambda_{n}}^{\left(  n-j\right)  }\left(  x\right)
\end{align*}
where we set $a_{j}=%
{\displaystyle\sum_{s+r=j}}
p_{r}p_{s}.$ By Theorem \ref{ThmMain1} we obtain
\[
p^{t}H_{k}p>%
{\displaystyle\sum_{j=0}^{2k}}
a_{j}x^{j}=%
{\displaystyle\sum_{j=0}^{2k}}
{\displaystyle\sum_{s+r=j}}
p_{r}p_{s}x^{j}=\left(  p_{0}+p_{1}x+\cdots+p_{k}x^{k}\right)  ^{2}.
\]

\textbf{Proof of Theorem \ref{ThmMain3}:} Without loss of generality we may
assume by using a translaton $t=x-a$ that the left point of the interval $a$
is equal to $0.$ We use Theorem 2.6.3 (due to M. Riesz) in \cite{Akhi65}. Let
$U_{n}$ be the space of all polynomials of degree $\leq n$. We define a linear
functional $\mathfrak{S}:U_{n}\rightarrow\mathbb{R}$ by
\[
\mathfrak{S}\left(  t^{k}\right)  :=\int_{0}^{b}k!\Phi_{\Lambda_{n}}^{\left(
n-k\right)  }\left(  x\right)  d\mu\left(  x\right)
\]
for $k=0,...,n$ and define $s_{k}=$ $\mathfrak{S}\left(  t^{k}\right)  $ for
$k=0,....,n$. We claim that the sequence $\left\{  s_{k}\right\}  _{k=0}^{n}$
is non-negative with respect to the interval $\left(  0,b\right)  $ (cf.
Definition 2.6.3 in \cite{Akhi65}). Let $R\left(  t\right)  =%
{\displaystyle\sum_{j=0}^{n}}
a_{j}t^{j}$ be a polynomial of degree $\leq n$ which is non-negative on
$\left(  0,b\right)  ,$ i.e. $R\left(  t\right)  \geq0$ for all $t\in\left(
0,b\right)  .$ Then%
\[
\mathfrak{S}\left(  R\right)  =%
{\displaystyle\sum_{k=0}^{n}}
a_{k}\int_{0}^{b}k!\Phi_{\Lambda_{n}}^{\left(  n-k\right)  }\left(  x\right)
d\mu\left(  x\right)  .
\]
By Theorem \ref{ThmMain1} the function $%
{\displaystyle\sum_{k=0}^{n}}
a_{k}k!\Phi_{\Lambda_{n}}^{\left(  n-l\right)  }\left(  x\right)  $ is
non-negative on $\left[  0,b\right]  $ and it follows that $\mathfrak{S}%
\left(  R\right)  \geq0.$ By Theorem 2.6.3 in \cite{Akhi65} there exists a
non-negative measure $\nu$ over $\left[  0,b\right]  $ such that
\[
s_{k}=\int_{0}^{b}t^{k}d\sigma\left(  t\right)  \text{ for }k=0,....,n.
\]
The proof is complete.

The following simple example will show that the statement of Theorem
\ref{ThmMain2} (and therefore of Theorem \ref{ThmMain1}) is not true if we
omit assumption $\Phi_{n}^{\left(  n+1\right)  }\left(  x\right)  \geq0$ for
all $x\in\left[  0,B\right]  .$

\begin{example}
\label{Example0}Let $\Lambda_{1}=\left(  -1,-2\right)  ,$ then $\Phi
_{\Lambda_{1}}\left(  x\right)  =e^{-x}-e^{-2x}.$ A computation shows that
\[
\det H_{1,1}\left(  x\right)  =\det\left(
\begin{array}
[c]{cc}%
\Phi_{\Lambda_{1}}^{\prime\prime}\left(  x\right)  & \Phi_{\Lambda_{1}%
}^{\prime}\left(  x\right) \\
\Phi_{\Lambda_{1}}^{\prime}\left(  x\right)  & 2\Phi_{\Lambda_{1}}\left(
x\right)
\end{array}
\right)  =e^{-2x}\left(  4e^{-2x}-6e^{-x}+1\right)
\]
which is negative for $x\in\left[  0,\ln\left(  \sqrt{5}+3\right)  \right]  $
and it changes its sign at $\ln\left(  \sqrt{5}+3\right)  =1.\,\allowbreak
655\,6$. Moreover, the second derivative $\Phi_{\Lambda_{1}}^{\prime\prime
}\left(  x\right)  =e^{-x}-4e^{-2x}$ is negative for $x\in\left[
0,2\ln2\right]  $ and it changes its sign at $2\ln2=\allowbreak1.\,\allowbreak
386\,3.$
\end{example}

In our main results we have required that $x>0.\ $One can extend the results
to the case $x<0$ by imposing an additional assumption:

\begin{theorem}
\label{ThmMain12plus} Let $\Lambda_{n}=\left(  \lambda_{0},\ldots,\lambda
_{n}\right)  \in\mathbb{C}^{n+1}$ and $A<0<B.$ Assume that $\Phi_{\Lambda_{n}%
}$ is real-valued and
\[
\Phi_{\Lambda_{n}}^{\left(  n+1\right)  }\left(  x\right)  \leq0\text{ for all
}x\in\left[  A,0\right]  \text{ and }\Phi_{\Lambda_{n}}^{\left(  n+1\right)
}\left(  x\right)  \geq0\text{ for all }x\in\left[  0,B\right]  .
\]
Then for any polynomial $R\left(  x\right)  =%
{\displaystyle\sum_{k=0}^{n}}
a_{k}x^{k}$ which is non-negative on the interval $\left[  A,B\right]  $ the
inequality
\[%
{\displaystyle\sum_{k=0}^{n}}
a_{k}k!\Phi_{\Lambda_{n}}^{\left(  n-k\right)  }\left(  x\right)  \geq
R\left(  x\right)
\]
holds for $x\in\left[  A,B\right]  .$ Moreover, for any natural number $k\leq
n/2$ the Hankel matrix $H_{n,k}\left(  x\right)  $ defined in Theorem
\ref{ThmMain2} is positive definite for $x\in\left[  A,B\right]  .$
\end{theorem}

\begin{proof}
Since $\Phi_{\Lambda_{n}}^{\left(  n+1\right)  }\left(  x\right)  \leq0$ for
all $x\in\left[  A,0\right]  $ and $R\left(  x\right)  \geq0$ for all
$x\in\left[  A,B\right]  $ we see that
\[
\int_{0}^{x}R\left(  t\right)  \Phi_{\Lambda_{n}}^{\left(  n+1\right)
}\left(  x-t\right)  dt=-\int_{x}^{0}R\left(  t\right)  \Phi_{\Lambda_{n}%
}^{\left(  n+1\right)  }\left(  x-t\right)  dt\geq0
\]
for any $x\in\left[  A,0\right]  .$ Now the statement for $x<0$ is obvious
from Theorem \ref{ThmMain0}. The proof of Theorem \ref{ThmMain2} for $x<0$
follows the same lines as before.
\end{proof}

\section{Inequalities for the fundamental function}

In our results we have assumed that $\Phi_{\Lambda_{n}}^{\left(  n+1\right)
}$ is real-valued. It is natural to ask under which conditions the fundamental
function $\Phi_{\Lambda_{n}}$ is real-valued. As usual we say that the space
$E_{\Lambda_{n}}$ is \emph{closed under complex conjugation}, if for $f\in$
$E_{\Lambda_{n}}$ the complex conjugate function $\overline{f}$ is again in
$E_{\Lambda_{n}}.$ It is easy to see that for complex numbers $\lambda_{_{0}%
},...,\lambda_{_{n}}$ the space $E_{\Lambda_{n}}$ is closed under complex
conjugation if and only if there exists a permutation $\sigma$ of the indices
$\left\{  0,...,n\right\}  $ such that $\overline{\lambda_{_{j}}}%
=\lambda_{_{\sigma\left(  j\right)  }}$ for $j=0,...,n.$ In other words,
$E_{\Lambda_{n}}$ is closed under complex conjugation if and only if the
vector $\Lambda_{_{n}}=\left(  \lambda_{_{0}},...,\lambda_{_{n}}\right)  $ is
equal up to reordering to the conjugated vector $\overline{\Lambda_{_{n}}%
}:=\left\{  \overline{\lambda_{0}},...,\overline{\lambda_{n}}\right\}  .$ The
following result is easy to prove and its proof is omitted.

\begin{proposition}
\label{Prop1}If $\Lambda_{_{n}}\ $is equal up to reordering to $\overline
{\Lambda_{_{n}}}$ then $\Phi_{\Lambda_{_{n}}}\left(  t\right)  $ is a
real-valued for any $t\in\mathbb{R}.$
\end{proposition}

Recall that a $E_{\Lambda_{n}}$ is an extended Chebyshev space over an
interval $I$ if each non-zero $f\in E_{\Lambda_{n}}$ has at most $n$ zeros
(including multiplicities) in $I$. The following fact is well known, see e.g.
\cite{KounchevRender2021JCAM}

\begin{proposition}
\label{Prop2}If $\Lambda_{_{n}}=\left(  \lambda_{_{0}},...,\lambda_{_{n}%
}\right)  \in\mathbb{R}^{n+1}$ then $E_{\Lambda_{n}}$ is an extended Chebyshev
space over $\mathbb{R}$ and $\Phi_{\Lambda_{n}}\left(  t\right)  >0$ for all
$t>0$.
\end{proposition}

The following result is probably mathematical folkore and we include a proof
for convenience of the reader.

\begin{theorem}
\label{ThmMicc} Assume that $\Lambda_{_{n}}\ $is equal up to reordering to
$\overline{\Lambda_{_{n}}}$ . If $E_{\Lambda_{n}}$ is an extended Chebyshev
space on $\left[  0,B\right)  $ then $\Phi_{\Lambda_{n}}\left(  x\right)  >0$
for all $x\in\left(  0,B\right)  $ and $\frac{\Phi_{\Lambda_{n}}^{\prime}%
}{\Phi_{\Lambda_{n}}}$ is decreasing on $\left(  0,B\right)  $ and
\[
\Phi_{\Lambda_{n}}^{\prime\prime}\left(  x\right)  \Phi_{\Lambda_{n}}\left(
x\right)  \leq\left(  \Phi_{\Lambda_{n}}^{\prime}\left(  x\right)  \right)
^{2}\text{ for all }x\in\left(  0,B\right)  .
\]
and
\end{theorem}

\begin{proof}
Note that $\Phi_{\Lambda_{n}}$ has exactly $n$ zeros at $x=0.$ Since
$E_{\Lambda_{n}}$ is an extended Chebyshev space on $\left[  0,B\right)  $ it
can not have any zero in $\left(  0,B\right)  ,$ hence $\Phi_{\Lambda_{n}%
}\left(  x\right)  >0$ for all $x\in\left(  0,B\right)  .$ Clearly $\frac
{\Phi_{\Lambda_{n}}^{\prime}}{\Phi_{\Lambda_{n}}}$ is decreasing on $\left(
0,B\right)  $ if
\[
w=\Phi_{\Lambda_{N}}^{\prime\prime}\Phi_{\Lambda_{N}}-\Phi_{\Lambda_{N}%
}^{\prime}\Phi_{\Lambda_{N}}^{\prime}.
\]
is negative on $\left(  0,B\right)  .$ Suppose that $w\left(  x_{0}\right)
=0$ for some $x_{0}\in\left(  0,B\right)  .$ Define%
\[
f\left(  x\right)  :=\Phi_{\Lambda_{N}}^{\prime\prime}\left(  x_{0}\right)
\Phi_{\Lambda_{N}}\left(  x\right)  -\Phi_{\Lambda_{N}}^{\prime}\left(
x_{0}\right)  \Phi_{\Lambda_{N}}^{\prime}\left(  x\right)  .
\]
Then $f\in E_{\Lambda_{N}}$ and $f\left(  x_{0}\right)  =w\left(
x_{0}\right)  =0.$ Further%
\[
f^{\prime}\left(  x\right)  =\Phi_{\Lambda_{N}}^{\prime}\left(  x_{0}\right)
\Phi_{\Lambda_{N}}^{\prime\prime}\left(  x\right)  -\Phi_{\Lambda_{N}}%
^{\prime\prime}\left(  x_{0}\right)  \Phi_{\Lambda_{N}}^{\prime}\left(
x\right)  .
\]
Thus $f^{\prime}\left(  x_{0}\right)  =0.$ Clearly $f$ has $N-1$ zeros at
$x=0,$ so $f$ has $N+1$ zeros in $\left[  0,B\right)  .$ Since $E_{\Lambda
_{n}}$ is an extended Chebyshev space on $\left[  0,B\right)  $ this implies
that $f$ is identical zero. Thus
\[
\Phi_{\Lambda_{N}}^{\prime\prime}\left(  x_{0}\right)  \Phi_{\Lambda_{N}%
}\left(  x\right)  =\Phi_{\Lambda_{N}}^{\prime}\left(  x_{0}\right)
\Phi_{\Lambda_{N}}^{\prime}\left(  x\right)
\]
for all $x\in\left[  0,B\right)  .$ Since $\Phi_{\Lambda_{N}}\left(  x\right)
$ and $\Phi_{\Lambda_{N}}^{\prime}\left(  x\right)  $ are linearly independent
this implies that $\Phi_{\Lambda_{N}}^{\prime}\left(  x_{0}\right)  =0$ and
$\Phi_{\Lambda_{N}}^{\prime\prime}\left(  x_{0}\right)  =0.$ It follows that
$\Phi_{\Lambda_{N}}^{\prime}$ has a double zero at $x_{0}$ and $N-1$ zeros at
$0,$ so it has $N+1$ zeros, so $\Phi_{\Lambda_{N}}^{\prime}$ is identical
zero, a contradiction. Thus $w$ has only zeros at $x=0,$so it does not change
the sign. Looking at the Taylor series of $w$ it is easy to see that $w\left(
x\right)  $ is negative for $x>0$ sufficiently small. This finishes the proof.
\end{proof}

Our results lead to an upper estimate of the function $x\longmapsto\left(
\Phi_{\Lambda_{n}}^{\prime}\left(  x\right)  \right)  ^{2}$:

\begin{theorem}
\label{ThmMain4}Assume that $\Lambda_{_{n}}\ $is equal up to reordering to
$\overline{\Lambda_{_{n}}}$ and $B>0.$ Assume that $\Phi_{\Lambda_{n}%
}^{\left(  n+1\right)  }\left(  x\right)  \geq0$ for all $x\in\left[
0,B\right]  .$ Then for any $x>0.$
\begin{equation}
\left(  \Phi_{\Lambda_{n}}^{\prime}\left(  x\right)  \right)  ^{2}<\frac
{n}{n-1}\Phi_{\Lambda_{n}}^{\prime\prime}\left(  x\right)  \Phi_{\Lambda_{n}%
}\left(  x\right)  . \label{turanineq1}%
\end{equation}

\end{theorem}

\begin{proof}
By Theorem \ref{ThmMain2} we know that the submatrix
\[
A_{n}:=\left(
\begin{array}
[c]{cc}%
\left(  n-2\right)  !\Phi_{\Lambda_{n}}^{\prime\prime}\left(  x\right)  &
\left(  n-1\right)  !\Phi_{\Lambda_{n}}^{\prime}\left(  x\right) \\
\left(  n-1\right)  !\Phi_{\Lambda_{n}}^{\prime}\left(  x\right)  &
n!\Phi_{\Lambda_{n}}\left(  x\right)
\end{array}
\right)
\]
is positive definite, so it has a positive determinant and
\[
\det A_{n}=\left(  n-1\right)  !\left(  n-1\right)  !\left(  \frac{n}{n-1}%
\Phi_{\Lambda_{n}}^{\prime\prime}\left(  x\right)  \Phi_{\Lambda_{n}}\left(
x\right)  -\left(  \Phi_{\Lambda_{n}}^{\prime}\left(  x\right)  \right)
^{2}\right)  >0.
\]

\end{proof}

In the next example we show that inequality (\ref{turanineq1}) does not hold
in general for $x<0.$

\begin{example}
\label{Example1}We take $\Lambda_{3}=\left(  -1,1,0,1\right)  $. Then it is
easy to verify that
\[
\Phi_{\Lambda_{3}}\left(  x\right)  =\frac{1}{2}\left(  e^{x}\left(
x-2\right)  +\sinh x+2\right)  .
\]
Then $\frac{d^{4}}{dx^{4}}\Phi_{\Lambda_{3}}\left(  x\right)  =e^{x}+\frac
{1}{2}\sinh x+\frac{1}{2}xe^{x}>0$ for $x>0.$ Theorem \ref{ThmMicc} provides
the lower bound $1$ in
\[
1\leq F\left(  x\right)  =\frac{\Phi_{\Lambda_{3}}^{\prime}\left(  x\right)
\Phi_{\Lambda_{3}}^{\prime}\left(  x\right)  }{\Phi_{\Lambda_{3}}%
^{\prime\prime}\left(  x\right)  \Phi_{\Lambda_{3}}\left(  x\right)  }%
=\frac{2\left(  \frac{1}{2}\cosh x-\frac{1}{2}e^{x}+\frac{1}{2}xe^{x}\right)
^{2}}{\left(  \frac{1}{2}\sinh x+\frac{1}{2}xe^{x}\right)  \left(
e^{x}\left(  x-2\right)  +\sinh x+2\right)  }%
\]
for all $x>0.$ Theorem \ref{ThmMain4} for $n=3$ shows that $F\left(  x\right)
$ is bounded by $3/2$ whenever $x\geq0.$ The graph of $F$ shows that $F\left(
x\right)  $ takes its maximum for some $x_{0}<0,$ and the maximum value is
larger than $3/2.$ Alternatively, one can check that $F\left(  -0.6\right)
=1.\,\allowbreak5387>3/2.$ In this example, the submatrix $A_{n}$ is not
positive definite for $x<0,$ so the matrix defined in Theorem \ref{ThmMain2}
is not positive definite for $x<0.$
\end{example}

\section{Monotonicity of the derivatives of the fundamental function}

We have seen that our assumption, $\Phi_{\Lambda_{n}}^{\left(  n+1\right)
}\left(  x\right)  \geq0$ for all $x>0$, is essential in our theorems. Hence
it is natural to discuss monotonicity properties of the fundamental function
in dependence of the vector $\Lambda_{n}.$

\begin{proposition}
\label{Prop3}Let $\Lambda_{_{n}}=\left(  \lambda_{_{0}},...,\lambda_{_{n}%
}\right)  \in\mathbb{R}^{n+1}$ and $n\geq2.$ Then $\Phi_{\Lambda_{n}}^{\prime
}\left(  x\right)  >0$ for all $x>0$ if and only if there exists $j\in\left\{
0,...,n\right\}  $ such that $\lambda_{j}$ is non-negative.
\end{proposition}

\begin{proof}
For sufficiency, assume without loss of generality that $\lambda_{n}\geq0.$ It
is easy to see (and well known) that
\[
\left(  \frac{d}{dx}-\lambda_{n}\right)  \Phi_{\left(  \lambda_{0}%
,...,\lambda_{n}\right)  }=\Phi_{\left(  \lambda_{0},...,\lambda_{n-1}\right)
}.
\]
By Proposition 1, $\Phi_{\left(  \lambda_{0},...,\lambda_{n}\right)  }\left(
x\right)  >0$ and $\Phi_{\left(  \lambda_{0},...,\lambda_{n-1}\right)
}\left(  x\right)  >0$ for all $x>0$ and it follows that
\[
\Phi_{\left(  \lambda_{0},...,\lambda_{n}\right)  }^{\prime}\left(  x\right)
=\Phi_{\left(  \lambda_{0},...,\lambda_{n-1}\right)  }\left(  x\right)
+\lambda_{n}\Phi_{\left(  \lambda_{0},...,\lambda_{n}\right)  }\left(
x\right)  \geq0.
\]
For necessity, assume that all $\lambda_{j}$ are negative and we are aiming
for a contradiction. Let $\lambda_{1},...,\lambda_{r}$ be the pairwise
different frequencies and $d_{1},...,d_{r}$ the multiplicities. The
fundamental function is of the form
\[
\Phi_{\Lambda_{n}}\left(  x\right)  =%
{\displaystyle\sum_{j=0}^{r}}
{\displaystyle\sum_{l_{j}=0}^{d_{j}-1}}
c_{j,l_{j}}x^{l_{j}}e^{\lambda_{j}x}.
\]
Since $\lambda_{j}<0$ for all $j=0,...,r,$ it follows that $\Phi_{\Lambda_{n}%
}\left(  x\right)  \rightarrow0$ for $x\rightarrow\infty.$ Further
$\Phi_{\Lambda_{n}}\left(  0\right)  =0$ since $n\geq2.$ Thus $\Phi
_{\Lambda_{n}}$ has an absolute maximum on $\left(  0,\infty\right)  $ and
therefore it has a local maximum, hence there exists $x_{0}$ with
$\Phi_{\Lambda_{n}}^{\prime}\left(  x_{0}\right)  =0,$ a contradiction.
\end{proof}

\begin{proposition}
\label{Prop4}Let $\Lambda_{_{n}}=\left(  \lambda_{_{0}},...,\lambda_{_{n}%
}\right)  \in\mathbb{R}^{n+1}.$ If there exist $j,k\in\left\{
0,....,n\right\}  $ with $k\neq j$ and $\lambda_{j}+\lambda_{k}\geq0$ then%
\[
\Phi_{\Lambda_{n}}^{\prime}\left(  x\right)  >0\text{ and }\Phi_{\Lambda_{n}%
}^{\prime\prime}\left(  x\right)  >0\text{ for all }x>0.
\]

\end{proposition}

\begin{proof}
If $\lambda_{j}+\lambda_{k}\geq0$ then $\lambda_{j}$ or $\lambda_{k}$ is
$\geq0.$ Then $\Phi_{\Lambda_{n}}^{\prime}\left(  x\right)  >0$ by Proposition
\ref{Prop3}. For the next statement we may assume without loss of generality
that $\lambda_{j}=\lambda_{n}$ and $\lambda_{k}=\lambda_{n-1}.$ Using the
notation $D_{\lambda_{n}}f=f^{\prime}-\lambda_{n}f$ we consider%
\[
D_{\lambda_{n}}D_{\lambda_{n-1}}\Phi_{\Lambda_{n}}=\Phi_{\Lambda_{n}}%
^{\prime\prime}-\left(  \lambda_{n}+\lambda_{n-1}\right)  \Phi_{\Lambda_{n}%
}^{\prime}+\lambda_{n}\lambda_{n-1}\Phi_{\Lambda_{n}}.
\]
Note that
\[
D_{\lambda_{n}}D_{\lambda_{n-1}}\Phi_{\Lambda_{n}}=D_{\lambda_{n}}%
D_{\lambda_{n-1}}\Phi_{\left(  \lambda_{0},...,\lambda_{n}\right)  }%
=\Phi_{\left(  \lambda_{0},...,\lambda_{n-2}\right)  }=:\Phi_{\Lambda_{n-2}}.
\]
Using that $D_{\lambda_{n}}\Phi_{\Lambda_{n}}=\Phi_{\Lambda_{n}}^{\prime
}-\lambda_{n}\Phi_{\Lambda_{n}}$ and $D_{\lambda_{n}}\Phi_{\Lambda_{n}}%
=\Phi_{\Lambda_{n-1}}$ and we see that
\begin{align*}
\Phi_{\Lambda_{n-2}}  &  =\Phi_{\Lambda_{n}}^{\prime\prime}-\left(
\lambda_{n}+\lambda_{n-1}\right)  \left(  \Phi_{\Lambda_{n-1}}+\lambda_{n}%
\Phi_{\Lambda_{n}}\right)  +\lambda_{n}\lambda_{n-1}\Phi_{\Lambda_{n}}\\
&  =\Phi_{\Lambda_{n}}^{\prime\prime}-\lambda_{n}^{2}\Phi_{\Lambda_{n}%
}-\left(  \lambda_{n}+\lambda_{n-1}\right)  \Phi_{\Lambda_{n}-1}\left(
x\right)  .
\end{align*}
It follows that%
\begin{equation}
\Phi_{\Lambda_{n}}^{\prime\prime}=\Phi_{\Lambda_{n-2}}+\lambda_{n}^{2}%
\Phi_{\Lambda_{n}}+\left(  \lambda_{n}+\lambda_{n-1}\right)  \Phi
_{\Lambda_{n-1}}\left(  x\right)  . \label{seconddeerid}%
\end{equation}
By Proposition \ref{Prop2} the functions on the right hand side are
non-negative for $x>0.$
\end{proof}

Note that the identity (\ref{seconddeerid}) can be used to discuss derivatives
of order $3$ and $4$: if $\left(  \lambda_{0},...,\lambda_{n-2}\right)  $
contains a pair $\lambda_{j},\lambda_{k}$ such that $\lambda_{j}+\lambda
_{k}\geq0$ then we can apply the proven statement to the functions on the
right hand side of (\ref{seconddeerid}) and it follows that $\Phi_{\Lambda
_{n}}^{\prime\prime\prime}\left(  x\right)  >0$ and $\Phi_{\Lambda_{n}%
}^{\left(  4\right)  }\left(  x\right)  >0$ for $x>0.$ It is obvious that one
can repeat this process if $\left(  \lambda_{0},...,\lambda_{n-4}\right)  $
contains a pair of frequencies $\lambda_{k}$ and $\lambda_{j}$ with
$\lambda_{k}+\lambda_{j}\geq0.$ This is the case if the vector $\Lambda_{n}$
is \emph{symmetric} which means there exists a permutation $\pi$ of the set
$\left\{  0,...,n\right\}  $ such that $-\lambda_{_{j}}=\lambda_{_{\pi(j)}}$
for $j=0,...,n,$ or, symbolically, $-\Lambda_{n}=\Lambda_{n}.$ Thus we have

\begin{corollary}
Let $\Lambda_{n}=\left(  \lambda_{0},...,\lambda_{n}\right)  \in
\mathbb{R}^{n+1}$ be symmetric. Then $\Phi_{\Lambda_{n}}^{\left(  k\right)
}\left(  x\right)  \geq0$ for all $x\geq0$ and all natural numbers $k.$
\end{corollary}

Finally we observe the following necessary condition for the non-negativity of
$\Phi_{\Lambda_{n}}^{\left(  n+1\right)  }\left(  x\right)  .$

\begin{theorem}
If $\Phi_{\Lambda_{n}}^{\left(  n+1\right)  }\left(  x\right)  \geq0$ for all
$x>0$ then $\lambda_{0}+\cdots+\lambda_{n}\geq0.$
\end{theorem}

\begin{proof}
The Taylor series of $\Phi_{\Lambda_{n}}\left(  x\right)  $ can be computed
explicitly (see \cite{KoRe13}): we know $\Phi_{\Lambda_{n}}^{\left(  n\right)
}\left(  0\right)  =1$ and for $k\geq n$ formula
\begin{equation}
\Phi_{\Lambda_{n}}^{\left(  k\right)  }\left(  0\right)  =\sum
_{\substack{\left(  s_{0},\ldots,s_{n}\right)  \in\mathbb{N}_{0}^{n+1}\text{
}\\s_{0}+\cdots+s_{n}+n=k,}}\lambda_{0}^{s_{0}}\cdots\lambda_{n}^{s_{n}}
\label{eqTaylorcoeffphi}%
\end{equation}
holds. For $k=0$, one obtains $\Phi_{\Lambda_{n}}^{\left(  n+1\right)
}\left(  0\right)  =\lambda_{0}+\cdots+\lambda_{n}.$ If $\Phi_{\Lambda_{n}%
}^{\left(  n+1\right)  }\left(  0\right)  <0,$ the first of the Taylor series
of $\Phi_{\Lambda_{n}}^{\left(  n+1\right)  }\left(  x\right)  $ is negative,
hence $\Phi_{\Lambda_{n}}^{\left(  n+1\right)  }\left(  x\right)  <0$ for
$x>0$ sufficiently small, a contradiction.
\end{proof}

\end{document}